\documentclass[a4paper]{amsart}

\usepackage[dvipsnames]{xcolor} 
\usepackage{amsmath,amsthm,amssymb}
\usepackage{enumitem}
\usepackage{tikz}
\usepackage{hyperref}
\usepackage{amsfonts,marvosym,amscd}
\usepackage{mathtools}
\usepackage{mathscinet}
\usepackage{wasysym} 
\usepackage{graphicx}
\usepackage{psfrag}
\usepackage{datetime}
\usepackage{tikz-cd}  

\usetikzlibrary{decorations.pathmorphing}
\usetikzlibrary{positioning}
\usetikzlibrary{calc}
\usetikzlibrary{fit}

\numberwithin{equation}{section}

\newtheorem{theorem}{Theorem}[section]
\newtheorem{lemma}[theorem]{Lemma}
\newtheorem{proposition}[theorem]{Proposition}
\newtheorem{corollary}[theorem]{Corollary}
\newtheorem{observation}[theorem]{Observation}

\newtheorem{innercustomthm}{Theorem}
\newenvironment{customthm}[1]
  {\renewcommand\theinnercustomthm{#1}\innercustomthm}
  {\endinnercustomthm}

\theoremstyle{definition}
\newtheorem{definition}[theorem]{Definition}

\newtheorem{example}[theorem]{Example}

\theoremstyle{remark}
\newtheorem{remark}[theorem]{Remark}

\newcommand{\nonconsec}[2]{\prescript{\circlearrowleft}{}{\mathbf{I}_{#1}^{#2}}}
\newcommand{\derset}[2]{{\tilde{\mathbf{I}}_{#1}^{#2}}}
\newcommand{\walk}{\dashrightarrow}

\let\emptyset\varnothing 

\newcommand{\st}{\,\middle\vert\,} % For `such that' in sets
\newcommand{\set}[1]{\left\{\,#1\,\right\}} % For making a set with a condition. (To get the spacing right)
\newcommand{\onevec}{\mathbf{1}} % A vector of ones
\newcommand{\ivec}{E_{i}} % Vector with 1 in the i-th coordinate and zero elsewhere
\newcommand{\rvec}{r\ivec} % For adding r to the i-th coordinate etc.
\newcommand{\sv}[1]{sE_{#1}} % For one you want to change the subscript
\newcommand{\svec}{s\ivec}

\newcommand{\tqdn}{\tilde{Q}^{(d, n)}}
\newcommand{\qdn}{Q^{(d, n)}}

\newcommand{\seg}[2]{\overline{#1#2}} % Line segment between two points

\newcommand{\darcs}[1]{\mathsf{d}\text{-}\mathsf{arcs}(#1)} % Set of d-arcs of a triangulation

\newcommand{\swr}{\taurus} % Symmetric intertwining

\DeclareMathOperator{\conv}{conv} % Convex hull

\makeatletter
\newcommand\@dotsep{4.5}
\def\@tocline#1#2#3#4#5#6#7{\relax
  \ifnum #1>\c@tocdepth % then omit
  \else
    \par \addpenalty\@secpenalty\addvspace{#2}%
    \begingroup \hyphenpenalty\@M
    \@ifempty{#4}{%
      \@tempdima\csname r@tocindent\number#1\endcsname\relax
    }{%
      \@tempdima#4\relax
    }%
    \parindent\z@ \leftskip#3\relax \advance\leftskip\@tempdima\relax
    \rightskip\@pnumwidth plus1em \parfillskip-\@pnumwidth
    #5\leavevmode\hskip-\@tempdima{#6}\nobreak
    \leaders\hbox{$\m@th\mkern \@dotsep mu\hbox{.}\mkern \@dotsep mu$}\hfill
    \nobreak
    \hbox to\@pnumwidth{\@tocpagenum{\ifnum#1=1\fi#7}}\par% <-- \bfseries for \section page
    \nobreak
    \endgroup
  \fi}
\AtBeginDocument{%
\expandafter\renewcommand\csname r@tocindent0\endcsname{0pt}
}
\def\l@subsection{\@tocline{2}{0pt}{2.5pc}{5pc}{}}
\makeatother

\usepackage[backend=biber,
	style=alphabetic,
	url=false,
	doi=false,
	isbn=false,
	maxbibnames=99]{biblatex}
\DeclareFieldFormat{title}{#1} % Removes quotation marks around title
\renewbibmacro{in:}{} % Removes ``In:'' before journal etc
\title{Quiver combinatorics for higher-dimensional triangulations}
\author{Nicholas J. Williams}
\email{williams@math.uni-koeln.de}
\address{Abteilung Mathematik, Department Mathematik/Informatik der Universit\"at zu K\"oln, Weyertal 86-90, 50931 Cologne, Germany}
\thanks{This paper forms part of my PhD studies. I would like to thank my supervisor Professor Sibylle Schroll for her continuing help and attention, and Hugh Thomas for interesting conversations. I would also like to thank the Isaac Newton Institute for Mathematical Sciences, Cambridge, for support and hospitality during the programme \textit{Cluster algebras and representation theory} where work on this paper was undertaken. This work was supported by EPSRC grant no EP/K032208/1.}

\begin{document}

\begin{abstract}
We investigate the combinatorics of quivers that arise from triangulations of even-dimensional cyclic polytopes. Work of Oppermann and Thomas pinpoints such quivers as the prototypes for higher-dimensional cluster theory. We first show that a $2d$-dimensional triangulation has no interior $(d + 1)$-simplices if and only if its quiver is a cut quiver of type $A$, in the sense of Iyama and Oppermann. This is a higher-dimensional generalisation of the fact that triangulations of polygons with no interior triangles correspond to orientations of an $A_{n}$ Dynkin diagram. An application of this first result is that the set of triangulations of a $2d$-dimensional cyclic polytope with no interior $(d + 1)$-simplices is connected via bistellar flips---the higher-dimensional analogue of flipping a diagonal inside a quadrilateral. In dimensions higher than 2, bistellar flips cannot be performed at all locations in a triangulation. Our second result gives a quiver-theoretic criterion for performing bistellar flips on a triangulation of a $2d$-dimensional cyclic polytope. This provides a visual tool for studying mutability of higher-dimensional triangulations and points towards what a theory of higher-dimensional quiver mutation could look like. Indeed, we apply this result to give a rule for mutating cut quivers at vertices which are not necessarily sinks or sources.
\end{abstract}

\maketitle

\tableofcontents

\section{Introduction}

In this paper we investigate the combinatorics of quivers associated to triangulations of even-dimensional cyclic polytopes. As shown by Oppermann and Thomas \cite{ot}, such quivers provide the prototype for higher-dimensional cluster theory.

Since their introduction \cite{fz1}, cluster algebras have generated substantial amounts of fruitful research touching many areas of mathematics, including dynamical systems \cite{fz-laurent,cs_cube,speyer_oct}, Poisson geometry \cite{gsv03,gsv05}, and Teichm\"uller theory \cite{fg06,fg09}. A cluster algebra can be given by choosing a quiver with a variable assigned to each vertex, and subsequently generating new quivers by a process of mutation, with new variables given from old variables via ``exchange relations''. A remarkable result is that the cluster algebras of finite cluster type are precisely those coming from Dynkin diagrams, just as in the Cartan--Killing classification \cite{fz2}. Two particular topics which are connected with cluster algebras are surfaces \cite{fst} and representation theory of algebras \cite{ccs,bmrrt}. Surfaces can be used to produce cluster algebras via triangulations, whilst in representation theory, cluster algebras are categorified via cluster categories.

A particularly simple example of a cluster algebra coming from a surface is the cluster algebra of type $A_{n}$, where the clusters are in bijection with triangulations of a convex $(n + 3)$-gon \cite{fz1}. Mutation of clusters corresponds to flipping a diagonal inside a quadrilateral. The quiver of a cluster can be easily constructed from the triangulation by drawing arrows between neighbouring arcs. In the type $A$ cluster category, these quivers are the Gabriel quivers of the endomorphism algebras of the corresponding cluster-tilting objects.

An active area of research within representation theory is higher Auslander--Reiten theory, which was introduced by Iyama in a series of papers \cite{iy-high-ar,iy-aus,iy-clus}. This subject studies subcategories of module categories which behave like higher-dimensional versions of abelian categories, and shows that several classical results can be generalised to them. Higher Auslander--Reiten theory has found beautiful connections with other areas of mathematics, such as symplectic geometry \cite{djl}, non-commutative algebraic geometry \cite{himo}, and combinatorics \cite{ot,njw-hst}.

Cluster theory and higher Auslander--Reiten theory were connected in \cite{ot}, where higher-dimensional cluster phenomena were discovered for type $A$. Here higher-dimensional cluster categories were defined, with the cluster-tilting objects corresponding to triangulations of even-dimensional cyclic polytopes. Mutation of clusters now corresponds to bistellar flips---the higher-dimensional analogue of flipping a polygon inside a quadrilateral.

Higher-dimensional cluster theory remains poorly understood. A higher cluster algebra has yet to be defined---if such a definition is indeed possible. It would be remarkable if cluster algebras were the two-dimensional instance of a more general phenomenon. The necessary ingredients for a higher cluster algebra would be a rule for quiver mutation and an exchange relation to produce new cluster variables after mutation. Whilst higher tropical exchange relations were exhibited in \cite{ot}, it is known that na\"ively detropicalising these relations does not work.

Triangulations of even-dimensional cyclic polytopes present themselves as the guide for how the higher-dimensional quiver combinatorics ought to work. Just as in the classical type $A$ case, the quiver of a cluster arises both from the endomorphism algebra of the corresponding cluster-tilting object, and may also be constructed from the corresponding triangulation: the vertices of the quiver correspond to the internal $d$-simplices of the triangulation, which we refer to as \emph{$d$-arcs}, with arrows between the $d$-arcs that are closest to each other.

We investigate the information encoded by the quiver associated to a triangulation of a $2d$-dimensional cyclic polytope. The best-understood quivers are those known as \emph{cut quivers}, which were introduced in \cite{io}. These quivers have a rule for mutation at sinks and sources \cite{io}. For $d = 1$, these cut quivers are precisely orientations of the $A_{n}$ Dynkin diagram. Our first result shows that cut quivers correspond precisely to triangulations with no interior $(d + 1)$-simplices and that these are also exactly the triangulations whose quivers are acyclic. Hence, this is a higher-dimensional generalisation of the fact that a triangulation of a convex polygon has no internal triangles if and only if its quiver is acyclic, and in this case the quiver is an orientation of the $A_{n}$ Dynkin diagram.

\begin{customthm}{A}[Theorem~\ref{thm:interior}]\label{thm:intro:interior}
A triangulation of a $2d$-dimensional cyclic polytope has no interior $(d + 1)$-simplices if and only if its quiver is acyclic, in which case its quiver is a cut quiver of type $A$.
\end{customthm}

An application of this result is that the set of triangulations of a $2d$-dimensional cyclic polytope without internal $(d + 1)$-simplices is connected via bistellar flips.

Unlike for the 2-dimensional case, for $d > 1$ it is not possible to perform a bistellar flip at every internal $d$-simplex of a $2d$-dimensional triangulation, or, equivalently, at every vertex of its quiver. This is an important difference with classical cluster theory, where a key property is that one can mutate a given cluster at every vertex of its quiver. This feature makes higher-dimensional cluster theory much more difficult to work with. Our second result uses the quiver of a triangulation to give a combinatorial criterion for identifying which $d$-simplices are \emph{mutable}---that is, admit a bistellar flip. We show how the arrows in the quiver can be partitioned into paths which we call \emph{maximal retrograde paths}, and prove the following theorem.

\begin{customthm}{B}[Theorem~\ref{thm:retro}]\label{thm:intro:criterion}
Let $\mathcal{T}$ be a triangulation of $C(n + 2d + 1,2d)$. A $d$-arc of $\mathcal{T}$ is mutable if and only if it is not in the middle of a maximal retrograde path.
\end{customthm}

This theorem gives a quiver-theoretic criterion for mutability, and hence points towards what a theory of higher-dimensional quiver mutation \cite{fz2} could look like. Other extensions of quiver mutation have been of interest in the literature, such as to ice quivers \cite{press_frozen}. Moreover, this provides a visual way of understanding mutability for higher-dimensional triangulations, and makes it easier to compute bistellar flips of higher-dimensional triangulations by hand. In the case of polygon triangulations, all retrograde paths are of length one, so that the criterion imposes no restriction and all arcs are mutable. As an application of this theorem, we give a rule for mutating cut quivers at vertices which are not necessarily sinks or sources.

This paper is structured as follows. In Section~\ref{sect:back} we give background to the paper, predominantly on triangulations of even-dimensional cyclic polytopes. In Section~\ref{sect:triangs} we consider triangulations of $2d$-dimensional cyclic polytopes without interior $(d + 1)$-simplices and prove Theorem~\ref{thm:intro:interior}. In Section~\ref{sect:mut} we study mutability of $d$-simplices in terms of quivers and prove Theorem~\ref{thm:intro:criterion}.

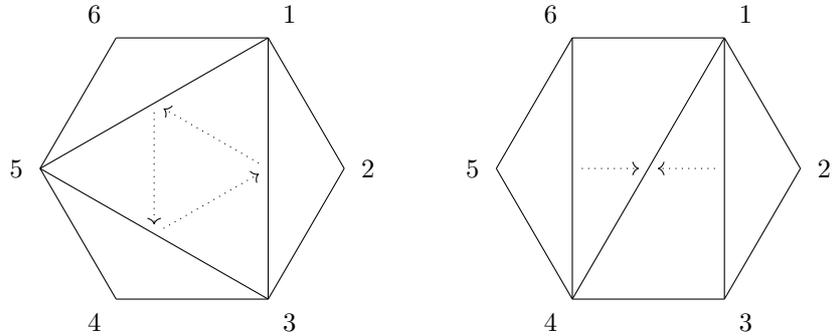
\begin{figure}
\caption{Triangulations of polygons with and without interior triangles}\label{fig:int_triang}
\[
\begin{tikzpicture}

\begin{scope}[shift={(-3,0)}]

% Internal triangulation
\coordinate(2) at (0:2);
\node at (2) [right = 1mm of 2] {2};
\coordinate(1) at (60:2);
\node at (1) [above right = 1mm of 1] {1};
\coordinate(6) at (120:2);
\node at (6) [above left = 1mm of 6] {6};
\coordinate(5) at (180:2);
\node at (5) [left = 1mm of 5] {5};
\coordinate(4) at (240:2);
\node at (4) [below left = 1mm of 4] {4};
\coordinate(3) at (300:2);
\node at (3) [below right = 1mm of 3] {3};

\node(13) at ($(1)!0.5!(3)$) {};
\node(35) at ($(3)!0.5!(5)$) {};
\node(51) at ($(5)!0.5!(1)$) {};

\draw (1) -- (2);
\draw (2) -- (3);
\draw (3) -- (4);
\draw (4) -- (5);
\draw (5) -- (6);
\draw (6) -- (1);

\draw (1) -- (3);
\draw (3) -- (5);
\draw (1) -- (5);

\draw[->, dotted] (13) -- (51);
\draw[->, dotted] (51) -- (35);
\draw[->, dotted] (35) -- (13);

\end{scope}

%--------------------------------------

\begin{scope}[shift={(3,0)}]

% Fan triangulation
\coordinate(2) at (0:2);
\node at (2) [right = 1mm of 2] {2};
\coordinate(1) at (60:2);
\node at (1) [above right = 1mm of 1] {1};
\coordinate(6) at (120:2);
\node at (6) [above left = 1mm of 6] {6};
\coordinate(5) at (180:2);
\node at (5) [left = 1mm of 5] {5};
\coordinate(4) at (240:2);
\node at (4) [below left = 1mm of 4] {4};
\coordinate(3) at (300:2);
\node at (3) [below right = 1mm of 3] {3};

\node(13) at ($(1)!0.5!(3)$) {};
\node(14) at ($(1)!0.5!(4)$) {};
\node(64) at ($(6)!0.5!(4)$) {};

\draw (1) -- (2);
\draw (2) -- (3);
\draw (3) -- (4);
\draw (4) -- (5);
\draw (5) -- (6);
\draw (6) -- (1);

\draw (1) -- (3);
\draw (1) -- (4);
\draw (6) -- (4);

\draw[->, dotted] (13) -- (14);
\draw[->, dotted] (64) -- (14);

\end{scope}

\end{tikzpicture}
\]
\end{figure}

\section{Background}\label{sect:back}

\subsection{Conventions}

We use $[m]$ to denote the set $\{1, \dots, m\}$. By $\binom{[m]}{k}$ we mean the set of subsets of $[m]$ of size $k$. For $l \in [m]$, using the notation of \cite{ops}, we use $<_{l}$ to denote the cyclically shifted order on $[m]$ given by \[l <_{l} l+1 <_{l} < \dots <_{l} m-1 <_{l} m <_{l} 1 <_{l} \dots <_{l} l-1 .\] For $r \geqslant 3$, $a_{1} < \dots < a_{r}$ is a \emph{cyclic ordering} if there is an $l \in [n + 2d + 1]$ such that $a_{1} <_{l} \dots <_{l} a_{r}$.

Throughout this paper, we will often need to change the ordering on $[m]$ to $<_{l}$ for some $l$. We will usually do this tacitly, rather than explicitly writing $<_{l}$ in place of $<$. Similarly, we will usually be tacitly using arithmetic modulo $n + 2d + 1$, where this will be the number of vertices of our cyclic polytope. Hence, we shall write things like $j = i - 1$, when we mean $j \equiv i - 1 \pmod{n + 2d + 1}$. Our convention here will also be that our equivalence-class representatives for arithmetic modulo $n + 2d + 1$ will be $[n + 2d + 1]$ rather than $\{0, 1, \dots, n + 2d\}$. This is because we shall label the vertices of our cyclic polytope by $[n + 2d + 1]$.

Finally, given a tuple $A \in [m]^{k + 1}$, we will denote the elements of $A$ by $(a_{0}, a_{1}, \dots, a_{k})$ where $a_{0} <_{l} a_{1} <_{l} \dots <_{l} a_{k}$ in a cyclically shifted order to be determined by the context. By default this will be the usual order on $[m]$. Because we wish sometimes to change this order, we are also tacitly identifying tuples which are the same up to cyclically shifted reordering. The same applies to other letters of the alphabet: the upper case letter denotes the tuple; the lower case letter is used for the entries, which are ordered according to their index, which starts from 0.

\subsection{Triangulations of cyclic polytopes}

Cyclic polytopes should be thought of as higher-dimensional analogues of convex polygons. General introductions to this class of polytopes can be found in \cite[Lecture 0]{ziegler}, \cite[4.7]{gruenbaum}, \cite[Section 6.1]{lrs} and \cite[Chapter VI]{bar}. Let $n \geqslant 0$ and $\delta \geqslant 1$ and consider the curve defined by $p_{t}=(t, t^{2}, \dots , t^{\delta}) \subset \mathbb{R}^{\delta}$ for $t \in \mathbb{R}$, which is known as the \emph{moment curve}. Choose $t_{1}, t_{2}, \dots , t_{n + \delta + 1} \in \mathbb{R}$ such that $t_{1} < t_{2} < \dots < t_{m}$. The convex hull $\conv\{ p_{t_{1}}, p_{t_{2}}, \dots , p_{t_{n + \delta + 1}} \}$ is a \emph{cyclic polytope} $C(n + \delta + 1, \delta)$. A \emph{(geometric) triangulation} of a cyclic polytope $C(n + \delta + 1, \delta)$ is a collection of (geometric) $\delta$-simplices whose interiors are pairwise disjoint and whose union is $C(n + \delta + 1, \delta)$.

A \emph{facet} of $C(n + \delta + 1,\delta)$ is a face of codimension one. A \emph{circuit} of a cyclic polytope $C(n + \delta + 1,\delta)$ is a pair, $(Z_{+}, Z_{-})$, of sets of vertices of $C(n + \delta + 1,\delta)$ which are inclusion-minimal with respect to the property $\mathrm{conv}(Z_{+}) \cap \mathrm{conv}(Z_{-}) \neq \emptyset$.

The circuits and facets of a cyclic polytope are independent of its particular geometric realisation given by the choice of set of points on the moment curve, by \cite{gale,breen}. This implies that whether or not a collection of $(\delta+1)$-subsets of $[n + \delta + 1]$ forms a triangulation of $C(n + \delta + 1, \delta)$ is likewise independent of the particular geometric realisation. Hence, in this paper we primarily consider triangulations combinatorially: a \emph{combinatorial triangulation} of $C(n + \delta + 1, \delta)$ is a collection of ordered $(\delta + 1)$-tuples with entries in $[n + \delta + 1]$ which gives a geometric triangulation of $C(n + \delta + 1, \delta)$. Unless otherwise specified, when we write `triangulation' we mean a combinatorial triangulation. Likewise, if we say that $A$ is a $k$-simplex, we mean that $A$ is an ordered $(k + 1)$-tuple with entries in $[n + \delta + 1]$. We write $|A|$ to denote the geometric realisation of $A$ as a geometric simplex which is the convex hull of $k + 1$ points on the moment curve. 

In this paper we are exclusively interested in triangulations of even-dimensional cyclic polytopes. These were given an elegant combinatorial description in \cite{ot}, which we now explain. This combinatorial description was in turn used to connect triangulations of even-dimensional cyclic polytopes to representation theory of algebras. We will sometimes comment on these connections, but this paper does not require the reader to be familiar with representation theory.

A $d$-simplex of a triangulation of $C(n + 2d + 1, 2d)$ is \emph{internal} if it does not lie within a facet of $C(n + 2d + 1, 2d)$. It is clear that a triangulation of a convex polygon is determined by the arcs of the triangulation; similarly, a triangulation of $C(n + 2d + 1,2d)$ is determined by the internal $d$-simplices of the triangulation, by a theorem of Dey \cite{dey}. For brevity, we refer to an internal $d$-simplex of a triangulation $\mathcal{T}$ of $C(n + 2d + 1, 2d)$ as a \emph{$d$-arc} of $\mathcal{T}$. We write $\darcs{\mathcal{T}}$ for the set of $d$-arcs of $\mathcal{T}$. We say that a $(d + 1)$-simplex of $\mathcal{T}$ is \emph{interior} if all of its facets are $d$-arcs.

In order to recall description of the triangulations of $C(n + 2d + 1,2d)$ from \cite{ot}, and for use later in the paper, we denote the following sets.
\begin{align*}
\derset{n + 2d + 1}{d} &= \set{(a_{0},\dots,a_{d}) \in \mathbb{Z}^{d+1} \st \parbox{5.5cm}{\begin{center}$\forall i \in \{0, 1, \dots ,d - 1\}, a_{i + 1} \geqslant a_{i} + 2 $\\  and  $a_{d} + 2 \leqslant a_{0} + n + 2d + 1$\end{center}}}; \\
\nonconsec{n + 2d + 1}{d} &= \derset{n + 2d + 1}{d} \cap [n + 2d + 1]^{d + 1}.
\end{align*}
Algebraically, $\derset{n + 2d + 1}{d}$ labels the $d$-cluster-tilting subcategory of the derived category of $A_{n}^{d}$, whilst $\nonconsec{n + 2d + 1}{d}$ labels its cluster category \cite{ot}.

A $d$-simplex $A \in [n + 2d + 1]^{d + 1}$ is a $d$-arc in $C(n + 2d + 1, 2d)$ if and only if $A \in \nonconsec{n + 2d + 1}{d}$. A $d$-simplex $A$ and a $d$-simplex $B$ are \emph{intertwining} if \[ a_{0} < b_{0} < a_{1} < b_{1} < \dots < a_{d} < b_{d}\] is a cyclic ordering, in which case we write $A \swr B$. Note that $A$ and $B$ being intertwining is therefore independent of the cyclically shifted order $<_{l}$ we may choose on $[n + 2d + 1]$. A collection of $d$-simplices is called \emph{non-intertwining} if no pair of its elements are intertwining. The circuits of the cyclic polytope $C(n + 2d + 1, 2d)$ are pairs of intertwining $d$-arcs.

There is a bijection between elements of $\nonconsec{n + 2d + 1}{d}$ and $d$-arcs of $C(n + 2d + 1, 2d)$, which induces a bijection between non-intertwining collections of $\binom{n + d - 1}{d}$ $d$-simplices in $\nonconsec{n + 2d + 1}{d}$ and triangulations of $C(n + 2d + 1, 2d)$ \cite[Theorem 2.3, Theorem 2.4]{ot}, via sending $\mathcal{T}$ to $\darcs{\mathcal{T}}$ (see also \cite[Remark 2.3]{njw-equal}).

Triangulations of cyclic polytopes can be \emph{mutated} by operations known as \emph{bistellar flips}. Triangulations $\mathcal{T}$ and $\mathcal{T}'$ of $C(n + 2d + 1,2d)$ are bistellar flips of each other if and only if $\darcs{\mathcal{T}}$ and $\darcs{\mathcal{T}'}$ have all but one element in common \cite[Theorem 4.1]{ot}.

\subsection{Cuts and slices}

We will later associate quivers to triangulations of even-dimensional cyclic polytopes. We will be particularly interested in when the quivers of triangulations take particular forms. Indeed, we now define the quivers which are higher analogues of orientations of the $A_{n}$ Dynkin diagram, following \cite{io}. Let $\qdn$ be the quiver with vertices \[Q_{0}^{(d,n)} := \set{(a_{0}, \dots, a_{d}) \in \mathbb{Z}_{\geqslant 0}^{d + 1} \st \sum_{i = 0}^{d}a_{i} = n - 1}\] and arrows \[Q_{1}^{(d, n)} := \set{A \to A + f_{i} \st A, A + f_{i} \in Q_{0}^{(d, n)}},\] where \[f_{i} = (\dots, 0, \overset{i}{-1}, \overset{i + 1}{1}, 0, \dots),\] with \[f_{d} = (1, 0, \dots, 0, -1).\] We say that arrows $A \to A + f_{i}$ are \emph{arrows of type $i$}. See Figure~\ref{fig:qdn_ex} for pictures of these quivers. A subset $C \subseteq Q_{1}^{(d, n)}$ is called \emph{cut} if it contains exactly one arrow from each $(d + 1)$-cycle in $Q^{(d, n)}$. Given a cut $C$, we write $Q_{C}^{(d, n)}$ for the quiver with arrows $Q_{1}^{(d, n)}\setminus C$ and refer to this as the \emph{cut quiver}. Examples of cut quivers can be seen in Figure~\ref{fig:qdn_cuts}. Note that the cut quivers of $Q^{(1, n)}$ are precisely the orientations of the $A_{n}$ Dynkin diagram. Hence, for $d > 1$, we think of cut quivers of $\qdn$ as higher analogues of orientations of the $A_{n}$ Dynkin diagram.

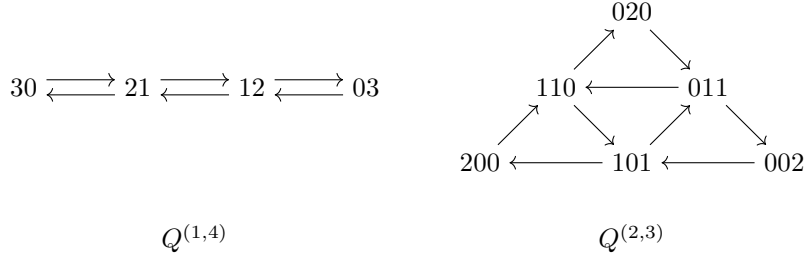
\begin{figure}
\caption{Examples of the quivers $Q^{(d, n)}$}\label{fig:qdn_ex}
\[
\begin{tikzpicture}

%%%%%%%%%%%%%%%%%%%%%%%%
% A_4
%%%%%%%%%%%%%%%%%%%%%%%%

\begin{scope}[shift={(-3,1)},xscale=1.5]

\node(30) at (0,0) {30};
\node(21) at (1,0) {21};
\node(12) at (2,0) {12};
\node(03) at (3,0) {03};

\draw[transform canvas={yshift=1mm},->] (30) -- (21);
\draw[transform canvas={yshift=-1mm},->] (21) -- (30);

\draw[transform canvas={yshift=1mm},->] (21) -- (12);
\draw[transform canvas={yshift=-1mm},->] (12) -- (21);

\draw[transform canvas={yshift=1mm},->] (12) -- (03);
\draw[transform canvas={yshift=-1mm},->] (03) -- (12);

\node at (1.5,-2) {$Q^{(1,4)}$};

\end{scope}

%%%%%%%%%%%%%%%%%%%%%%%%%%
% A_3^2
%%%%%%%%%%%%%%%%%%%%%%%%%%

\begin{scope}[shift={(3,0)}]

\node(200) at (0,0) {200};
\node(110) at (1,1) {110};
\node(020) at (2,2) {020};
\node(011) at (3,1) {011};
\node(002) at (4,0) {002};
\node(101) at (2,0) {101};

\draw[->] (200) -- (110);
\draw[->] (110) -- (020);
\draw[->] (020) -- (011);
\draw[->] (011) -- (002);
\draw[->] (002) -- (101);
\draw[->] (101) -- (200);
\draw[->] (110) -- (101);
\draw[->] (101) -- (011);
\draw[->] (011) -- (110);

\node at (2,-1) {$Q^{(2,3)}$};

\end{scope}

\end{tikzpicture}
\]
\end{figure}

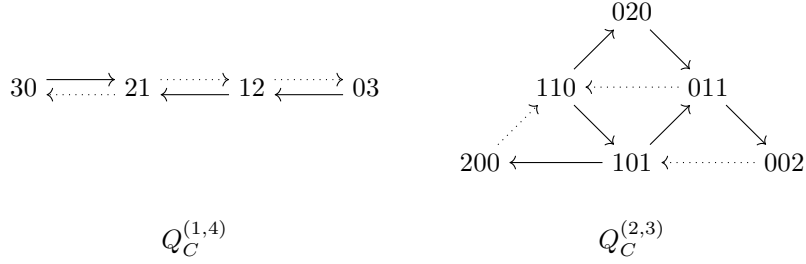
\begin{figure}
\caption{Cuts of the quivers $Q^{(d, n)}$}\label{fig:qdn_cuts}
\[
\begin{tikzpicture}

%%%%%%%%%%%%%%%%%%%%%%%%
% A_4
%%%%%%%%%%%%%%%%%%%%%%%%

\begin{scope}[shift={(-3,1)},xscale=1.5]

\node(30) at (0,0) {30};
\node(21) at (1,0) {21};
\node(12) at (2,0) {12};
\node(03) at (3,0) {03};

\draw[transform canvas={yshift=1mm},->] (30) -- (21);
\draw[transform canvas={yshift=-1mm},->,dotted] (21) -- (30);

\draw[transform canvas={yshift=1mm},->,dotted] (21) -- (12);
\draw[transform canvas={yshift=-1mm},->] (12) -- (21);

\draw[transform canvas={yshift=1mm},->,dotted] (12) -- (03);
\draw[transform canvas={yshift=-1mm},->] (03) -- (12);

\node at (1.5,-2) {$Q_{C}^{(1,4)}$};

\end{scope}

%%%%%%%%%%%%%%%%%%%%%%%%%%
% A_3^2
%%%%%%%%%%%%%%%%%%%%%%%%%%

\begin{scope}[shift={(3,0)}]

\node(200) at (0,0) {200};
\node(110) at (1,1) {110};
\node(020) at (2,2) {020};
\node(011) at (3,1) {011};
\node(002) at (4,0) {002};
\node(101) at (2,0) {101};

\draw[->,dotted] (200) -- (110);
\draw[->] (110) -- (020);
\draw[->] (020) -- (011);
\draw[->] (011) -- (002);
\draw[->,dotted] (002) -- (101);
\draw[->] (101) -- (200);
\draw[->] (110) -- (101);
\draw[->] (101) -- (011);
\draw[->,dotted] (011) -- (110);

\node at (2,-1) {$Q_{C}^{(2,3)}$};

\end{scope}

\end{tikzpicture}
\]
\end{figure}

Iyama and Oppermann show that cut quivers of $Q^{(d, n)}$ are precisely the quivers than can be realised as \emph{slices} of another family of quivers, denoted $\tqdn$, which we now define. Let $\tilde{Q}^{(d,n)}$ be the quiver with vertices \[\tilde{Q}_{0}^{(d,n)} := \derset{n+2d+1}{d}\] and arrows \[\tilde{Q}_{1}^{(d,n)} := \set{A \rightarrow A + \ivec \st A, A + \ivec \in \tilde{Q}_{0}^{(d,n)}},\] where \[\ivec:=(\dots, 0, \overset{i}{1}, 0, \dots).\] Given two $(d + 1)$-tuples $A, B \in \mathbb{Z}^{d + 1}$, we write $A \leqslant B$ if $a_{i} \leqslant b_{i}$ for all $i \in \{0, 1, \dots, d\}$. We use $A < B$ analogously. Note that, given $A, B \in \tilde{Q}_{0}^{(d, n)}$, there is a path $A \leadsto B$ in $\tilde{Q}^{(d, n)}$ if and only if $A \leqslant B$.

We denote by $\nu_{d}$ the automorphism of $\tilde{Q}^{(d,n)}$ given by $A \mapsto A - \onevec$. We denote by $\pi\colon \tqdn_{0} \rightarrow \nonconsec{n + 2d + 1}{d}$ the map given by \[A \mapsto \mathsf{sort}(\pi(a_{0}), \pi(a_{1}), \dots, \pi(a_{d})),\] where $\pi(a_{i}) := a_{i} \pmod{n + 2d + 1}$ and $\mathsf{sort}$ indicates that we should sort the tuple so that it is increasing. Note that the definition of $\derset{n + 2d + 1}{d}$ guarantees that $\pi(A)$ does not contain any repeated entries.

\begin{figure}
\caption{$\tilde{Q}^{(1,3)}$}\label{fig:q13}
\[
\scalebox{0.75}{
\begin{tikzpicture}

% Nodes for vertices
\node(zca) at (0,0){13};
\node(baz) at (0,2){04};
\node(aba) at (2,1){14};
\node(zcb) at (4,0){\color{blue}{24}};
\node(baa) at (4,2){15};
\node(abb) at (6,1){\color{blue}{25}};
\node(zcc) at (8,0){35};
\node(bab) at (8,2){\color{blue}{26}};
\node(abc) at (10,1){36};
\node(zcd) at (12,0){46};
\node(bac) at (12,2){37};

% Nodes for ellipses
\node at (-1,0){\dots};
\node at (-1,1){\dots};
\node at (-1,2){\dots};

\node at (13,0){\dots};
\node at (13,1){\dots};
\node at (13,2){\dots};

% Drawing arrows
\draw[->] (zca) -- (aba);
\draw[->] (aba) -- (baa);
\draw[->] (baz) -- (aba);
\draw[->] (aba) -- (zcb);
\draw[->] (baa) -- (abb);
\draw[->] (abb) -- (zcc);
\draw[->,blue] (zcb) -- (abb);
\draw[->,blue] (abb) -- (bab);
\draw[->] (bab) -- (abc);
\draw[->] (abc) -- (zcd);
\draw[->] (zcc) -- (abc);
\draw[->] (abc) -- (bac);

\end{tikzpicture}
}
\]
\end{figure}
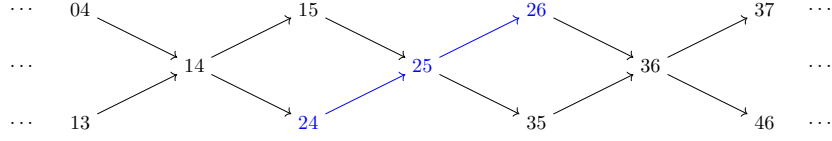

\begin{figure}
\caption{$\tilde{Q}^{(2, 3)}$}\label{fig:q23}
\[
\scalebox{0.7}{
\begin{tikzpicture}

% Nodes for vertices

%Layer 0
\node(zzcz) at (0,0){135};
\node(azbz) at (6,0){146};
\node(bzaz) at (12,0){157};
\node(zabz) at (3,1){136};
\node(aaaz) at (9,1){147};
\node(zbaz) at (6,2){137};

% Layer 1
\node(zzca) at (2,3){246};
\node(azba) at (8,3){257};
\node(bzaa) at (14,3){\color{blue}{268}};
\node(zaba) at (5,4){247};
\node(aaaa) at (11,4){\color{blue}{258}};
\node(zbaa) at (8,5){\color{blue}{248}};

% Layer 2
\node(zzcb) at (4,6){\color{blue}{357}};
\node(azbb) at (10,6){\color{blue}{368}};
\node(bzab) at (16,6){379};
\node(zabb) at (7,7){\color{blue}{358}};
\node(aaab) at (13,7){369};
\node(zbab) at (10,8){359};

% Nodes for ellipses
\node at (0,-1){\dots};
\node at (6,-1){\dots};
\node at (12,-1){\dots};

\node at (4,9){\dots};
\node at (10,9){\dots};
\node at (16,9){\dots};

% Drawing arrows

% Layer 0
\draw[->] (zzcz) -- (zabz);
\draw[->] (zabz) -- (azbz);
\draw[->] (azbz) -- (aaaz);
\draw[->] (aaaz) -- (bzaz);
\draw[->] (zabz) -- (zbaz);
\draw[->] (zbaz) -- (aaaz);

% Layer 1
\draw[->] (zzca) -- (zaba);
\draw[->] (zaba) -- (azba);
\draw[->] (azba) -- (aaaa);
\draw[->,blue] (aaaa) -- (bzaa);
\draw[->] (zaba) -- (zbaa);
\draw[->,blue] (zbaa) -- (aaaa);

% Layer 2
\draw[->,blue] (zzcb) -- (zabb);
\draw[->,blue] (zabb) -- (azbb);
\draw[->] (azbb) -- (aaab);
\draw[->] (aaab) -- (bzab);
\draw[->] (zabb) -- (zbab);
\draw[->] (zbab) -- (aaab);

% between layers
\draw[->] (azbz) -- (zzca);
\draw[->] (bzaz) -- (azba);
\draw[->] (aaaz) -- (zaba);

\draw[->] (azba) -- (zzcb);
\draw[->,blue] (bzaa) -- (azbb);
\draw[->,blue] (aaaa) -- (zabb);

\end{tikzpicture}
}
\]
\end{figure}
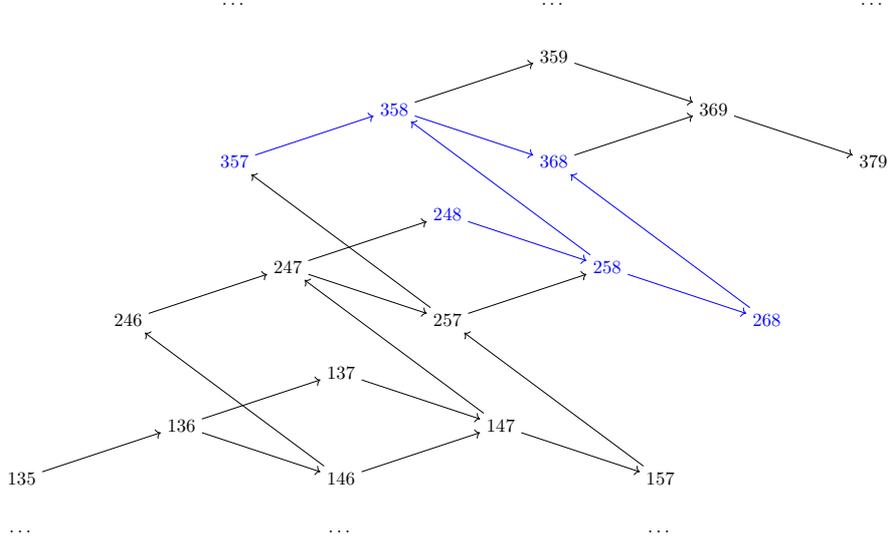

Following \cite[Definition 5.20]{io}, we define a \emph{slice} of $\tqdn$ to be a full subquiver $S$ of $\tqdn$ such that:
\begin{enumerate}
\item Any $\nu_{d}$ orbit in $\tqdn$ contains precisely one vertex which belongs to $S$.
\item $S$ is convex, i.e., for any path $P$ in $\tqdn$ connecting two vertices in $S$, all vertices appearing in $P$ belong to $S$.
\end{enumerate}
Slices are shown in blue in Figure~\ref{fig:q13} and Figure~\ref{fig:q23}.

The $\nu_{d}$-orbits of $\tqdn$ are in bijection with the vertices of $\qdn$. Given a slice $S$ of $\tqdn$, one can find a cut $C_{S}$ of $\qdn$ such that $Q_{C_{S}}^{(d, n)}$ is isomorphic to $S$ with the arrows $f_{i}$ of $Q_{C_{S}}^{(d, n)}$ corresponding to arrows $A \to A + E_{i}$ in $S$ \cite[Theorem 5.24]{io}.

\subsubsection{Mutation of cuts and slices}\label{sect:back:cut_slice:mut}

Cuts and slices can be mutated, as was defined in \cite{io}.
\begin{itemize}
\item Let $C$ be a cut of $\qdn$ and let $x$ be a source of $\qdn_{C}$. Define a subset $\mu_{x}^{+}(C)$ of $\qdn_{1}$ by removing all arrows in $C$ which end at $x$ and adding all arrows in $\qdn_{1}$ which begin at $x$. Dually, if $x$ is a sink of $\qdn_{C}$, define $\mu_{x}^{-}(C)$ by removing all arrows in $C$ which begin at $x$ and adding all arrows in $\qdn_{1}$ which end at $x$. By \cite[Proposition 5.14]{io}, we have that $\mu_{x}^{+}(C)$ and $\mu_{x}^{-}(C)$ are also cuts of $\qdn$.
\item Let $S$ be a slice of $\tqdn$. If $x$ is a source of $S$, then define a full subquiver $\mu_{x}^{+}(S)$ of $\tqdn$ by removing $x$ from $S$ and adding $\nu_{d}^{-1}x$ \cite[Definition 5.25]{io}. Dually, if $x$ is a sink of $S$, define a full subquiver $\mu_{x}^{-}(S)$ by removing $x$ and adding $\nu_{d}x$.
\end{itemize}
If $C_{S}$ is the cut corresponding to a slice $S$ then $C_{\mu_{x}^{+}{S}} = \mu_{x}^{+}(C_{S})$ and $C_{\mu_{x}^{-}{S}} = \mu_{x}^{-}(C_{S})$, provided $x$ is a source or sink, respectively. Here we abuse notation by using $x$ to refer both to the relevant vertex of $S$ and to the relevant vertex of $\qdn_{C_{S}}$.

\section{Triangulations without interior $(d+1)$-simplices}\label{sect:triangs}

In this section we prove our combinatorial description of triangulations of $C(n + 2d + 1, 2d)$ without interior $(d + 1)$-simplices and use this description to show that this class of triangulations is connected by bistellar flips.

\subsection{The quiver of a triangulation}

We first define the quiver of a triangulation, which is the higher-dimensional version of the quiver of a polygon triangulation---see, for instance, \cite[Section 3]{cdm-03} and \cite[Definition 2.12]{lkw-clus-intro}.

\begin{definition}\label{def:quiver}
We define the \emph{quiver} $Q(\mathcal{T})$ of $\mathcal{T}$ to be the following directed graph. The vertices are \[Q_{0}(\mathcal{T}) = \darcs{\mathcal{T}}.\] The arrows are given by $A \to B$ with $A \neq B$ such that $(A - \onevec) := (a_{0} - 1, \dots, a_{d} - 1) \swr B$ and there exists no $A' \in \darcs{\mathcal{T}}\setminus\{A, B\}$ with $(A - \onevec) \swr A'$ and $(A' - \onevec) \swr B$.
\end{definition}

We define the quiver in this way so that it coincides with the Gabriel quiver of the endmorphism algebra of the cluster-tilting object corresponding to the triangulation. The arrows mirror the description of the homomorphisms in the cluster category $\mathcal{O}_{A_{n}^{d}}$ from \cite{ot}. However, we now prove a simpler description of the quiver, given in Corollary~\ref{cor:quiv_desc}. In order to obtain this, we first make some observations about cyclic polytopes. Changing the ordering on on $[n + 2d + 1]$ from $<_{1}$ to $<_{l}$ induces a (combinatorial) automorphism of the cyclic polytope $C(n + 2d + 1, 2d)$---see \cite{kw}. We think of this automorphism as re-orientating the cyclic polytope. Generalising \cite{ot}, given a $2d$-simplex $T$ of $C(n + 2d + 1,2d)$ and an ordering $<_{l}$ of $[n + 2d + 1]$ under which $T$ is ordered $T = (t_{0}, t_{1}, \dots , t_{2d})$, we write $e_{l}(T) = (t_{0}, t_{2}, \dots , t_{2d-2}, t_{2d})$. If the ordering is determined by the context, we simply write $e(T)$.

\begin{lemma}\label{lem:find_2d_simplex}
Given an ordering $<_{l}$ of $[n + 2d + 1]$ and a $d$-arc $A$ of a triangulation $\mathcal{T}$ of $C(n + 2d + 1,2d)$, there is a $2d$-simplex $T$ of $\mathcal{T}$ such that $e_{l}(T) = A$.
\end{lemma}
\begin{proof}
This follows from applying \cite[Proposition 2.13]{ot} in the orientation given by $<_{l}$.
\end{proof}

\begin{proposition}\label{prop:all_but_one_vertex}
Suppose that $A \rightarrow B$ is an arrow in $Q(\mathcal{T})$. Then $A$ and $B$ share all but one entry, and there is a $2d$-simplex $T$ of $\mathcal{T}$ such that $A$ and $B$ are both faces of $T$.
\end{proposition}
\begin{proof}
We have that $(A - \onevec) \swr B$. It must be the case that $A$ and $B$ have at least one common entry, otherwise $A \swr B$. Hence, by re-ordering, we may assume that $a_{0} = b_{0}$ and $a_{d - 1} < a_{d} < b_{d}$. There must be a $2d$-simplex $T$ of $\mathcal{T}$ such that $e(T) = B$. We label the vertices of $T$ by $T = (b_{0}, x_{1}, b_{1}, x_{2}, \dots, b_{d - 1}, x_{d}, b_{d})$.

\sloppy We cannot have that $b_{i - 1} < x_{i} \leq a_{i} - 1$ for all $i \in [d]$, otherwise $A \swr (x_{1}, x_{2}, \dots, x_{d}, b_{d})$, which is impossible since $A$ and $(x_{1}, x_{2}, \dots, x_{d}, b_{d})$ are both $d$-arcs of $\mathcal{T}$. Hence there is an $i \in [d]$ such that $b_{i - 1} < a_{i} - 1 < x_{i}$. Then $(b_{0}, b_{1}, \dots, b_{i - 1}, x_{i}, b_{i + 1}, b_{i + 2}, \dots, b_{d})$ is a $d$-arc of $\mathcal{T}$, since $b_{i - 1} < a_{i} - 1 < x_{i}$ and $x_{i} < b_{i} < b_{i + 1}$. Moreover, $(A - \onevec) \swr (b_{0}, b_{1}, \dots, b_{i - 1}, x_{i}, b_{i + 1}, b_{i + 2}, \dots, b_{d})$ and $(b_{0} - 1, \dots, b_{i - 1} - 1, x_{i} - 1, b_{i + 1} - 1, \dots, b_{d} - 1) \swr B$. Since $A \rightarrow B$ is an arrow in $Q(\mathcal{T})$, we must have $A = (b_{0}, b_{1}, \dots, b_{i - 1}, x_{i}, b_{i + 1}, b_{i + 2}, \dots, b_{d})$. (In fact, by the ordering we have chosen, we must have that $i = d$.) Therefore $A$ and $B$ are both faces of the $2d$-simplex $T$ and they share all but one entry, as desired.
\end{proof}

\begin{corollary}\label{cor:quiv_desc}
The arrows of $Q(\mathcal{T})$ are \[Q_{1}(\mathcal{T}) = \set{A \rightarrow A + \rvec \st \parbox{6.5cm}{\begin{center}$A, A + \rvec \in \nonconsec{n + 2d + 1}{d},$ $\nexists s \in [r - 1] \text{ such that } A + \svec \in Q_{0}(\mathcal{T})$\end{center}}}.\]
\end{corollary}

We call arrows $A \to A + \rvec$ \emph{arrows of type $i$}. Note that this depends on the choice of an ordering $<_{l}$.

For an arrow $\alpha$, we denote the \emph{head} $h(\alpha)$ and the \emph{tail} $t(\alpha)$ such that $t(\alpha) \xrightarrow{\alpha} h(\alpha)$. We say that $\alpha$ is \emph{incident} at $t(\alpha)$ and $h(\alpha)$. Given a quiver $Q$ with vertices $A,B \in Q_{0}$, by a \emph{path} in $Q$ from $A$ to $B$ we mean a finite sequence of arrows $\alpha_{1}\dots\alpha_{r}$ such that $t(\alpha_{1}) = A, h(\alpha_{r}) = B$ and $h(\alpha_{i - 1}) = t(\alpha_{i})$ for all $i \in \{2, 3, \dots, s\}$. If there is a path from $A$ to $B$, then we write $A \leadsto B$. We note the following property concerning paths in $Q(\mathcal{T})$ which will be useful later.

\begin{lemma}\label{lem:find_path}
Given a triangulation $\mathcal{T}$ of $C(n + 2d + 1, 2d)$ and $A, B \in Q_{0}(\mathcal{T})$ with $(A - \onevec) \swr B$, there is a path from $A$ to $B$ in $Q(\mathcal{T})$.
\end{lemma}
\begin{proof}
This is clear from Definition~\ref{def:quiver}, using induction.
\end{proof}

The following property will be useful in Section~\ref{sect:mut}.

\begin{lemma}\label{lem:space}
Let $\mathcal{T}$ be a triangulation of $C(n + 2d + 1, 2d)$ with $A \in \darcs{\mathcal{T}}$. If $a_{i} + 2 < a_{i + 1}$ for some $i$, then there is either an arrow \[(a_{0}, \dots, a_{i}, a_{i + 1}^{-}, a_{i + 2}, \dots, a_{d}) \to A\] or an arrow \[A \to (a_{0}, \dots, a_{i - 1}, a_{i}^{+}, a_{i + 1}, \dots, a_{d}).\]
\end{lemma}
\begin{proof}
We use the fact that there is a $2d$-simplex $T$ of $\mathcal{T}$ such that $e(T) = A$. If we let $T = (a_{0}, x_{1}, a_{1}, x_{2}, \dots, a_{d - 1}, x_{d}, a_{d})$, then we must have either $x_{i + 1} > a_{i} + 1$ and $x_{i + 1} < a_{i + 1} - 1$. In the former case, we must have a path $(a_{0}, a_{1}, \dots, a_{i}, x_{i + 1}, a_{i + 2}, a_{i + 3}, \dots, a_{d}) \leadsto A$ comprised of arrows of type $i + 1$; in the latter case, we must have a path $A \leadsto (a_{0}, a_{1}, \dots, a_{i - 1}, x_{i}, a_{i + 1}, a_{i + 2} \dots, a_{d})$ of type $i$. This establishes the claim.
\end{proof}

\subsection{Quiver description}\label{sect:main}

With these preliminaries taken care of, we now move to prove the first main result of this paper, which describes triangulations $\mathcal{T}$ without interior $(d + 1)$-simplices in terms of their quivers $Q(\mathcal{T})$. Recall that a $(d + 1)$-simplex of a triangulation $\mathcal{T}$ of $C(n + 2d + 1, 2d)$ is interior if all of its facets are $d$-arcs.

Slices of $\tqdn$ give triangulations of $C(n + 2d + 1, 2d)$. We give a direct combinatorial proof of this, although it can also be deduced from the fact that slices correspond to iterated d-APR tilts \cite[Theorem 4.15]{io}, which implies that projecting to the cluster category will give a triangulation \cite[Theorem 6.4]{ot}. 

\begin{proposition}\label{prop:slice->triang}
If $S$ is a slice of $\tqdn$, then the vertices $\pi(S_{0})$ give a triangulation of $C(n + 2d + 1, 2d)$.
\end{proposition}
\begin{proof}
There are as many $\nu_{d}$-orbits as there are elements of $\nonconsec{n + 2d + 1}{d}$ containing $1$, namely $\binom{n + d - 1}{d}$. Suppose that there exist $\pi(A)$ and $\pi(B)$ in $\pi(S_{0})$ with $\pi(A) \swr \pi(B)$. We assume without loss of generality that $a_{0} < b_{0}$, noting that $\pi(A) \swr \pi(B)$ implies that $a_{i} \neq b_{i}$ for all $i$. We claim that $A < B$. Suppose for contradiction that $b_{i} < a_{i}$ for some $i$. We may choose the minimal $i$ such that this is the case. Then $a_{i - 1} < b_{i - 1} < b_{i} < a_{i}$. Since we must have \[\pi(a_{i - 1}) < \pi(b_{i - 1}) < \pi(a_{i}) < \pi(b_{i}),\] we must have $a_{i} - a_{i - 1} > n + 2d + 1$. Hence $a_{d} - a_{0} > n + 2d + 1 > n + 2d - 1$, which contradicts $A \in \derset{n + 2d + 1}{d}$.

Therefore $A < A + \onevec \leqslant B$, which means that here is then a path $A \leadsto A + \onevec \leadsto B$, so $A + \onevec \in \pi(S_{0})$ by convexity. But this contradicts the fact that $S$ contains one vertex from every $\nu_{d}$-orbit. Hence $\pi(S_{0})$ is a non-intertwining subset of $\nonconsec{n + 2d + 1}{d}$ of size $\binom{n + d - 1}{d}$, and so gives a triangulation of $C(n + 2d + 1, 2d)$.
\end{proof}

A similar argument also shows the following lemma, which will be useful later.

\begin{lemma}\label{lem:conv->slice}
If $S$ is a convex subquiver of $\tqdn$ such that $\pi(S_{0})$ is a triangulation of $C(n+2d+1,2d)$, then $S$ is a slice. 
\end{lemma}
\begin{proof}
Suppose that $S$ is a convex subquiver such that $\pi(S_{0})$ is a triangulation of $C(n + 2d + 1, 2d)$. Suppose for contradiction that $S$ possesses two vertices $A$ and $B$ which are in the same $\nu_{d}$-orbit. There is then either a path $A \leadsto B$ or a path $B \leadsto A$. Without loss of generality, we suppose the former. But then there is a path $A \leadsto A + \onevec \leadsto B$, so we must have $A + \onevec \in S_{0}$ by convexity. This is a contradiction, since $\pi(A)$ and $\pi(A + \onevec)$ are intertwining.
\end{proof}

The main theorem of this section is as follows.

\begin{theorem}\label{thm:interior}
A triangulation $\mathcal{T}$ of $C(n + 2d + 1, 2d)$ has no interior $(d+1)$-simplices if and only if its quiver is a cut of $\qdn$, and this is the case if and only if its quiver has no cycle.
\end{theorem}

Our strategy for proving this theorem is to prove facts about the quivers $Q(\mathcal{T})$ for triangulations $\mathcal{T}$ of $C(n + 2d + 1, 2d)$ without internal $(d + 1)$-simplices. We then use these properties to show that one can realise $Q(\mathcal{T})$ as a slice of $\tqdn$, which will imply that $Q(\mathcal{T})$ is a cut of $\qdn$. The idea of the proofs of both of the following lemmas is that if $Q(\mathcal{T})$ does not possess a certain property, then we can find an interior $(d + 1)$-simplex in~$\mathcal{T}$.

\begin{lemma}\label{lem:single-step-arrows}
If $\mathcal{T}$ has no interior $(d+1)$-simplices, then the arrows in $Q(\mathcal{T})$ are all of the form $A \rightarrow A + E_{i}$.
\end{lemma}
\begin{proof}
By Proposition~\ref{prop:all_but_one_vertex}, every arrow is of the form $A \rightarrow A + \rvec$ for some $r>0$ and for each such arrow, we have that $(a_{0}, a_{1}, \dots, a_{i}, a_{i} + r, a_{i+1}, a_{i + 2}, \dots, a_{d})$ is a face of a $2d$-simplex of $\mathcal{T}$. If $r>1$, then this is an interior $(d+1)$-simplex.
\end{proof}

\begin{lemma}\label{lem:simp_arrow_switch}
Suppose that $\mathcal{T}$ has no interior $(d+1)$-simplices. If $A \rightarrow A + E_{i} \rightarrow A + E_{i} + E_{j}$ is a sequence of arrows in $Q(\mathcal{T})$, then, if $A + E_{j} \in \nonconsec{m}{d}$, the sequence $A \rightarrow A + E_{j} \rightarrow A + E_{i} + E_{j}$ is also in $Q(\mathcal{T})$.
\end{lemma}
\begin{proof}
We assume $A + E_{j} \in \nonconsec{m}{d}$. By re-ordering, we may assume that $i = d$. Then there is a $2d$-simplex $T$ of $\mathcal{T}$ such that $e(T) = A + E_{j} + E_{d}$ by Lemma~\ref{lem:find_2d_simplex}.

Let $T = (a_{0}, x_{1}, a_{1}, \dots, x_{j}, a_{j} + 1, x_{j + 1}, \dots, x_{d}, a_{d} + 1)$. If $x_{d} = a_{d}$, then $A + E_{j}$ is a $d$-face of $T$ and hence a $d$-arc of $\mathcal{T}$. Hence, suppose for contradiction that $a_{d} \neq x_{d}$, so that $a_{d - 1} < x_{d} \leqslant a_{d} - 1$. Note that if $x_{j} = a_{j - 1} + 1$, then $(a_{0}, a_{1}, \dots , a_{d}) \swr (x_{1}, x_{2}, \dots ,x_{d}, a_{d} + 1)$, which is a $d$-face of $T$. Hence $a_{j - 1} + 2 \leqslant x_{j} < a_{j} + 1$. Then $(a_{0}, a_{1}, \dots, a_{j - 1}, x_{j}, x_{j + 1}, \dots , x_{d}, a_{d} + 1)$ is a $(d + 1)$-face of $T$ and an interior $(d+1)$-simplex of $\mathcal{T}$, a contradiction.
\end{proof}

We now prove some facts about the quivers $\tqdn$, which will be useful in proving the main theorem of this section. These will be used in combination with the previous two lemmas to show that one can realise $\mathcal{Q}(\mathcal{T})$ as a slice if $\mathcal{T}$ has no interior $(d + 1)$-simplices. We say that a full subquiver $P$ of $\tqdn$ is \emph{switching-closed} if whenever $A \rightarrow A + E_{i} \rightarrow A + E_{i} + E_{j}$ is a sequence of arrows in $P$ and $A + E_{j} \in \tqdn_{0}$, then the sequence of arrows $A \rightarrow A + E_{j} \rightarrow A + E_{i} + E_{j}$ is also in $P$.

\begin{lemma}\label{lem:all_paths}
Let $P$ be a switching-closed full subquiver of $\tqdn$. If there is a path $A \leadsto B$ in $P$, all other paths $A \leadsto B$ in $\tqdn$ must also lie in $P$.
\end{lemma}
\begin{proof}
Suppose that we have a path $A \leadsto B$ in $P$. The length of all such paths is $\sum_{i=0}^{d}(b_{i} - a_{i})$. We prove the claim by induction on this quantity. The base case, where the length is $1$, follows from the fact that $P$ is a full subquiver of $\tqdn$.

For the inductive step, we assume that the claim holds for all $X$ and $Y$ with $\sum_{i = 0}^{d}(y_{i} - x_{i}) < \sum_{i = 0}^{d}(b_{i} - a_{i})$. We have that $B$ is the head of up to $(d + 1)$ arrows, namely the ones with tails $B - E_{0}, B - E_{1}, \dots, B - E_{d}$, provided these are vertices of $\tqdn$. Suppose that $B - E_{i}$ is the penultimate vertex of our path $A \leadsto B$.

Choose a vertex $B - E_{j} \in \tqdn_{0}$, where $i \neq j$. If $A \nleq B - E_{j}$, then we may ignore this vertex, since there can be no paths from $A$ to $B$ through it. Hence we assume that $A \leqslant B - E_{j}$. This implies that $A \leqslant B - E_{i} - E_{j} \leqslant B - E_{i}$, so $B - E_{i} - E_{j} \in Q_{0}$ by the induction hypothesis applied to $A$ and $B - E_{i}$.

Since $P$ is switching-closed, we have that $B - E_{j} \in P_{0}$, since $B - E_{i} - E_{j}, B - E_{i}, B \in P_{0}$. By the induction hypothesis, all paths $A \leadsto B - E_{j}$ lie in $P$, and hence all paths $A \leadsto B$ passing through $B - E_{j}$ lie in $P$. The result follows.
\end{proof}

By a \emph{walk} in a quiver $Q$ from $A$ to $B$ we mean a finite sequence of arrows $\beta_{1}\beta_{2}\dots\beta_{s}$ such that $\beta_{1}$ is incident at $A$, $\beta_{s}$ is incident at $B$, and $\beta_{i-1}$ and $\beta_{i}$ are incident at a common vertex for all $i \in \{2, 3, \dots, s\}$. In this case, we write $A \walk B$. That is, a path only consists of forwards arrows, but a walk may contain backwards arrows as well.

\begin{lemma}\label{lem:walk_to_path}
Let $P$ be a connected switching-closed full subquiver of $\tqdn$. If $A,B \in P_{0}$ are such that there is a path $A \leadsto B$ in $\tqdn$, then there is a path $A \leadsto B$ in $P$.
\end{lemma}
\begin{proof}
Let $A,B \in P_{0}$. Suppose that there is a path $A \leadsto B$ in $\tqdn$. There is certainly a walk $W \colon A \walk B$ in $P$, since $P$ is connected. We prove that there is also a path $A \leadsto B$ in $P$ by induction on the number of backwards arrows in this walk. The base case, in which there are zero backwards arrows in the walk, is immediate.

Hence we suppose for induction that the claim holds for walks with fewer backwards arrows than $W$. We may assume that the final arrow in $W$ is a backwards one, otherwise we may remove the final arrow and consider instead the walk $A \walk B'$, where $A \walk B' \rightarrow B$ is the original walk~$W$.

Therefore we can assume that our walk is of the form $A \walk C \leftarrow B$, where $C \in P$. By the induction hypothesis, we can replace this with a walk of the form $A \leadsto C \leftarrow B$ in $P$. Moreover, by Lemma~\ref{lem:all_paths} we have that all paths $A \leadsto C$ in $\tqdn$ lie in $P$. Since we have a path $A \leadsto B$ in $\tqdn$, we have $A \leqslant B$ and, moreover, $A \leqslant B \leqslant C$. There is therefore a path $A \leadsto B \leadsto C$ in $\tqdn$. This path is in $P$, since every path $A \leadsto C$ is in $P$, which gives the desired path $A \leadsto B$ in~$P$.
\end{proof}

These two lemmas imply the following corollary.

\begin{corollary}\label{cor:switch+con->conv}
Let $P$ be a connected switching-closed full subquiver of $\tqdn$. Then $P$ is convex in $\tqdn$.
\end{corollary}
\begin{proof}
Suppose that $P$ is a connected switching-closed full subquiver of $\tqdn$. Let $A, B \in P_{0}$ be such that there is a path $A \leadsto B$ in $\tqdn$. By Lemma~\ref{lem:walk_to_path}, there is a path $A \leadsto B$ in $P$. Then, by Lemma~\ref{lem:all_paths}, we have that all paths $A \leadsto B$ lie in $P$, and so $P$ is convex.
\end{proof}

This corollary is useful because it is easier to check the property of being connected and switching-closed than the property of being convex. Given a full subquiver $P$ of $\tqdn$, we write $\overline{P}$ for the smallest switching-closed subquiver containing $P$. The subquiver $\overline{P}$ is well-defined since the intersection of a set of switching-closed full subquivers is switching-closed, so $\overline{P}$ may be constructed as the intersection of all switching-closed subquivers containing $P$.

\begin{proposition}\label{prop:no-cyc}
A triangulation $\mathcal{T}$ contains an interior $(d+1)$-simplex if and only if $Q(\mathcal{T})$ contains a cycle.
\end{proposition}
\begin{proof}
We first suppose for contradiction that $\mathcal{T}$ contains no interior $(d+1)$-simplices and that $Q(\mathcal{T})$ does contain a cycle. By Lemma~\ref{lem:single-step-arrows}, we can realise this cycle as a path $P \colon A \leadsto B$ in $\tqdn$, where $A$ is some vertex in the cycle in $Q(\mathcal{T})$ and $\pi(B) = A$ with $A \neq B$. We consider this path $P$ as a subquiver of $\tqdn$. By Lemma~\ref{lem:all_paths}, every path $A \leadsto B$ in $\tqdn$ must lie in $\overline{P}$. There is a path $A \leadsto A + \onevec \leadsto B$ in $\tqdn$. Hence $A + \onevec$ is a vertex of $\overline{P}$. By Lemma~\ref{lem:simp_arrow_switch}, if $C$ is a vertex of $\overline{P}$, then $\pi(C)$ is a vertex of $Q(\mathcal{T})$. Therefore $\pi(A + \onevec)$ is a vertex of $Q(\mathcal{T})$, but this is a contradiction, since $\pi(A + \onevec)$ and $A$ are intertwining.

We now suppose that $\mathcal{T}$ is a triangulation with an interior $(d + 1)$-simplex  $(a_{0}, \dots, a_{d + 1})$. Then $Q(\mathcal{T)}$ has a cycle given by concatenating the paths
\begin{align*}
(a_{0}, a_{1}, \dots, a_{d - 1}, a_{d}) &\leadsto (a_{0}, a_{1}, \dots , a_{d - 1}, a_{d + 1}) \leadsto (a_{0}, a_{1}, \dots, a_{d - 2}, a_{d}, a_{d + 1}) \leadsto \\
 \dots &\leadsto (a_{1}, a_{2}, \dots, a_{d}, a_{d + 1}) \leadsto (a_{0}, a_{1}, \dots , a_{d}),
\end{align*}
noting Lemma~\ref{lem:find_path}.
\end{proof}

\begin{remark}
For $d = 1$, an interior triangle gives a 3-cycle in the quiver where all the vertices of the cycle are edges of the triangle. For $d > 1$, the cycle obtained in Proposition~\ref{prop:no-cyc} may not exclusively have facets of the interior $(d + 1)$-simplex as its vertices. An example of this can be seen in Figure~\ref{fig:cyc_triang}, where the interior 3-simplex is 1357, but the cycle is $135 \to 136 \to 137 \to 147 \to 157 \to 357 \to 135$. Here $136$ and $147$ are not faces of the interior 3-simplex 1357.
\end{remark}

We are now ready to prove our first main result, noting that the second part of the statement has already been established by Proposition~\ref{prop:no-cyc}.

\begin{proof}[Proof of Theorem~\ref{thm:interior}]
First suppose that $\mathcal{T}$ has no interior $(d+1)$-simplices. We consider the full subquiver $R$ of $\tqdn$ with vertices \[R_{0} = \set{ B \in \tqdn_{0} \st \pi(B) \in Q_{0}(\mathcal{T})}.\] We claim that this is disconnected and that each connected component gives $Q(\mathcal{T})$ by applying $\pi$. Let $A \in Q_{0}(\mathcal{T})$. If $R$ is connected then it contains a walk $W \colon A \walk (a_{1}, a_{2}, \dots, a_{d}, a_{0} + n + 2d + 1)$. We consider $W$ as a subquiver of $\tqdn$ and consider $\overline{W}$. By Lemma~\ref{lem:walk_to_path}, $\overline{W}$ contains a path $P \colon A \leadsto (a_{1}, a_{2}, \dots, a_{d}, a_{0} + n + 2d + 1)$. By Lemma~\ref{lem:simp_arrow_switch}, if $B$ is a vertex of $\overline{W}$, then $\pi(B)$ is a vertex of $Q(\mathcal{T})$. Hence all vertices of $P$ give vertices of $Q(\mathcal{T})$, which therefore contains a cycle. But this contradicts Proposition~\ref{prop:no-cyc}.

By this argument, $R$ is disconnected and, moreover, the vertices of each connected component are in bijection with the vertices of $Q(\mathcal{T})$ via $\pi$, since $Q(\mathcal{T})$ is connected. Moreover, the arrows in each connected component of $R$ are the same as the arrows in $Q(\mathcal{T})$, by Lemma~\ref{lem:single-step-arrows}. Hence, by choosing one of the connected components, we obtain a full subquiver $S$ of $\tqdn$ such that $\pi(S) = Q(\mathcal{T})$. We then have that $S$ is a switching-closed connected subquiver of $\tqdn$, so $S$ is convex by Lemma~\ref{cor:switch+con->conv}. Since $\pi(S_{0})$ is a triangulation, it then follows from Lemma~\ref{lem:conv->slice} that $S$ is a slice. Hence $Q(\mathcal{T})$ is a cut of $\qdn$ by \cite[Theorem 5.24]{io}.

Now suppose that $Q(\mathcal{T})$ is a cut of $\qdn$. Then $Q(\mathcal{T})$ cannot contain any cycles, since cut quivers can be realised as slices, which are full subquivers of $\tqdn$, which does not contain any cycles. Then we obtain that $\mathcal{T}$ contains no interior $(d + 1)$-simplices by Proposition~\ref{prop:no-cyc}.
\end{proof}

Theorem~\ref{thm:interior} implies that the set of triangulations of $C(n + 2d + 1,2d)$ without interior $(d+1)$-simplices is connected by bistellar flips.

\begin{corollary}\label{cor:flips}
The class of triangulations of $C(n + 2d + 1, 2d)$ without interior $(d+1)$-simplices is connected by bistellar flips.
\end{corollary}
\begin{proof}
Slice mutation involves replacing one vertex of a slice $S$ with another to obtain a new slice $S'$. Hence, if one considers the triangulations $\pi(S_{0})$ and $\pi(S'_{0})$, these have all but one $d$-arc in common. Since two triangulations are related by a bistellar flip if and only if they have all but one $d$-arc in common \cite[Theorem 4.1]{ot}, it follows that triangulations related by slice mutation are related by a bistellar flip. Iyama and Oppermann then show that all slices are connected by slice mutation \cite[Theorem 5.27]{io}. Hence this implies that the class of  triangulations without interior $(d + 1)$-simplices is connected by bistellar flips, and so the result follows from Theorem~\ref{thm:interior}.
\end{proof}

\section{A combinatorial criterion for mutation}\label{sect:mut}

Given a triangulation $\mathcal{T}$ of $C(n + 2d + 1,2d)$, we say that \emph{a $d$-arc $A$ of $\mathcal{T}$ is mutable} if there is a bistellar flip $\mathcal{T}'$ of $\mathcal{T}$ such that $\darcs(\mathcal{T}) \setminus \{A\} = \darcs(\mathcal{T}')\setminus \{B\}$. It is clear that here $A$ and $B$ must intertwine. For $d = 1$, where triangulations of $C(n + 2d + 1, 2d)$ are triangulations of convex $m$-gons, all $d$-arcs are mutable. But this is not true for $d > 1$. In this section, we prove a criterion for identifying the mutable $d$-arcs of a triangulation $\mathcal{T}$ from its quiver $Q(\mathcal{T})$. This then leads us to a rule for mutating cut quivers at vertices which are neither sinks nor sources.

We begin with some motivating observations concerning cuts. We explain how a cut quiver may be decomposed into distinguished cut cycles, and observe that an arc of the triangulation is mutable if and only if it does not occur in the middle of a distinguished cut cycle. This will follow from the main result of the subsequent section.

It is clear that the arrows of each $(d + 1)$-cycle in $Q^{(d, n)}$ must be labelled exactly once by each element of $\{f_{0}, f_{1}, \dots, f_{d}\}$. We call a $(d + 1)$-cycle of $Q^{(d, n)}$ \emph{distinguished} if the arrows are labelled in the cyclic order $f_{d}, f_{d - 1}, \dots, f_{0}$. Given a cut $C$, the distinguished cut $(d + 1)$-cycles of $Q_{C}$ are the paths that result from removing the arrows of $C$ from the distinguished cycles of $\qdn$. 

\begin{observation}\label{obs:cut_mut}
Given a triangulation $\mathcal{T}$ whose quiver $Q(\mathcal{T})$ is a cut of $Q^{(d, n)}$, the mutable $d$-arcs of $\mathcal{T}$ are precisely the $d$-arcs which do not lie in the middle of a distinguished cut $(d + 1)$-cycle.
\end{observation}

\begin{example}
The reader can check that in the left-hand triangulation in Figure~\ref{fig:mut_via_cyc} the mutable 2-arcs are 135, 146, and 157, whilst in the right-hand triangulation the mutable 2-arcs are 246, 136, and 157. In these figures we draw each distinguished cut $3$-cycle in a different colour.
\end{example}

\begin{figure}
\caption{Mutability via distinguished cut $(d + 1)$-cycles}\label{fig:mut_via_cyc}
\[
\begin{tikzpicture}

\begin{scope}[shift={(-3,0)}]

\node(135) at (0,0) {$135$};
\node(136) at (1,1) {$136$};
\node(137) at (2,2) {$137$};
\node(146) at (2,0) {$146$};
\node(147) at (3,1) {$147$};
\node(157) at (4,0) {$157$};

\draw[red,->] (135) -- (136);
\draw[blue,->] (136) -- (137);
\draw[blue,->] (137) -- (147);
\draw[green,->] (147) -- (157);
\draw[red,->] (136) -- (146);
\draw[green,->] (146) -- (147);

\end{scope}

\begin{scope}[shift={(3,0)}]

\node(246) at (0,0) {$246$};
\node(136) at (1,1) {$136$};
\node(137) at (2,2) {$137$};
\node(146) at (2,0) {$146$};
\node(147) at (3,1) {$147$};
\node(157) at (4,0) {$157$};

\draw[red,->] (146) -- (246);
\draw[blue,->] (136) -- (137);
\draw[blue,->] (137) -- (147);
\draw[green,->] (147) -- (157);
\draw[red,->] (136) -- (146);
\draw[green,->] (146) -- (147);

\end{scope}

\end{tikzpicture}
\]
\end{figure}
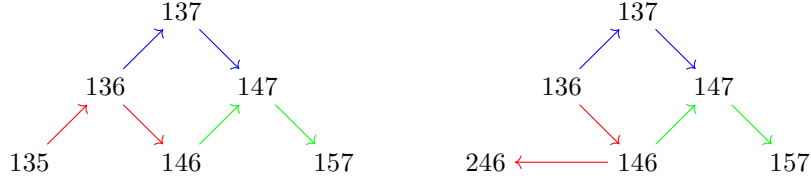

\subsection{General triangulations}\label{sect:mut:gen}

Cut quivers have a very particular form and it is this that allows us to determine the distinguished cut $(d + 1)$-cycles of the quiver, and then to use these to determine the mutable $d$-arcs of the triangulation. In general, quivers may be much more complicated than cut quivers. Nevertheless, we may generalise Observation~\ref{obs:cut_mut} to arbitrary triangulations of even-dimensional cyclic polytopes using the following notion.

\begin{definition}\label{def:retro}
Let $A$ be a $d$-arc of $\mathcal{T}$. A path \[A \rightarrow A + \rvec \rightarrow A + \rvec + \sv{i - 1}\] in $Q(\mathcal{T})$ is \emph{retrograde at $A + \rvec$} if $a_{i - 1} < a_{i - 1} + s < a_{i}$.

A path \[A_{0} \rightarrow A_{1} \rightarrow \dots \rightarrow A_{l} \rightarrow A_{l+1}\] is \emph{retrograde} if $A_{i-1} \rightarrow A_{i} \rightarrow A_{i+1}$ is retrograde at $A_{i}$ for all $i \in [l]$. For $i \in [l]$, we say that $A_{i}$ is \emph{in the middle} of this retrograde path. We consider paths consisting of a single arrow to be trivially retrograde. We say that a retrograde path is \emph{maximal} if it is not contained in any longer retrograde paths.
\end{definition}

\begin{remark}
It is hence clear that the distinguished cut $(d + 1)$-cycles in a cut quiver are maximal retrograde paths. Arrows in this cut cycle are labelled $f_{l + 1}, \dots, f_{d}, f_{0}, \dots, f_{l - 1}$ for some $l$. The arrows in $Q(\mathcal{T})$ labelled by $f_{i}$ are of the form $A \to A + E_{i}$, and we always have that $a_{i} + 1 < a_{i + 1}$. Hence, we obtain that these paths are retrograde. These paths are, futhermore, maximally retrograde, since the next arrow in the path at either end would have to be labelled by $f_{l}$, but this is precisely the arrow that has been cut out.
\end{remark}

\begin{lemma}
Every arrow in $Q(\mathcal{T})$ is contained in a unique maximal retrograde path.
\end{lemma}
\begin{proof}
If an arrow follows $A \to A + \rvec$, then must be of the form $A + \rvec \to A + \rvec + \sv{j}$ with $j = i - 1$ and if an arrow precedes it, then it must be of the form $A - \sv{j} \to A$ with $j = i + 1$. Both such arrows must be unique, by Corollary~\ref{cor:quiv_desc}. Hence, there is only one way to extend an arrow to a maximal retrograde path.
\end{proof}

\begin{proposition}
The maximal length of a retrograde path in $Q(\mathcal{T})$ is $d$.
\end{proposition}
\begin{proof}
Suppose for contradiction that we have a retrograde path in $Q(\mathcal{T})$ of length $(d+1)$. By re-ordering, we can represent this in the form \[(a_{0}, a_{1}, \dots, a_{d}) \rightarrow (a_{0}, a_{1}, \dots, a_{d - 1}, b_{d})\rightarrow \dots \rightarrow (a_{0}, b_{1}, b_{2}, \dots, b_{d}) \rightarrow (b_{0}, \dots, b_{d}).\] Since this path is retrograde, we have $a_{i} < b_{i} < a_{i + 1}$ for all $i \in [d]$. But this implies that $(a_{0}, a_{1}, \dots, a_{d})$ and $(b_{0}, b_{1}, \dots, b_{d})$ are intertwining.
\end{proof}

A $d$-arc $A$ is mutable precisely if there exists a $d$-arc $B$ which intertwines with it but which does not intertwine with any other $d$-arc in the triangulation. If such a $d$-arc $B$ does not exist, then $A$ is not mutable. Hence, we consider the collection of $d$-arcs of a triangulation which intertwine with a given $d$-arc outside the triangulation.

\begin{lemma}\label{lem:connected}
Let $B \in \nonconsec{m}{d} \setminus \darcs{\mathcal{T}}$. Let $Q_{B}(\mathcal{T})$ be the full subquiver of $Q(\mathcal{T})$ with vertex set \[Q_{B}(\mathcal{T})_{0} = \set{ A \in \darcs{\mathcal{T}} \st A \swr B }.\] Then $Q_{B}(\mathcal{T})$ is connected.
\end{lemma}
\begin{proof}
Let $\mathcal{T}_{B}$ be the collection of $2d$-simplices of $\mathcal{T}$ which have a $d$-face intertwining with the $d$-arc $B$. Let $T \in \mathcal{T}_{B}$. We first show that the set of $d$-faces of $T$ is connected in $Q_{B}(\mathcal{T})$. Hence, let $A, A'$ be two $d$-faces of $T$ which intertwine with $B$. Then $A$ and $A'$ must have a common vertex, since they are both faces of the same $2d$-simplex, so, by re-ordering, we can assume $a_{0}=a'_{0}$. We know that $A$ and $B$ must be intertwining, so we may also assume that \[a_{0} < b_{0} < a_{1} < b_{1} < \dots < b_{d-1} < a_{d} < b_{d}.\] Since $a_{0}=a'_{0}$ and $A'$ also intersects $B$, we also have that \[a'_{0} < b_{0} < a'_{1} < b_{1} < \dots < b_{d-1} < a'_{d} < b_{d}.\] Let $c_{i}, c'_{i} \in \{a'_{i}, a_{i}\}$ be such that $b_{i-1} < c_{i} < c'_{i} < b_{i}$. Then $C, C' \in \mathcal{T}$ since they are both $d$-faces of $T$. Moreover, they are both in $Q_{B}(\mathcal{T})$. There is a path $C \leadsto A$ in $Q_{B}(\mathcal{T})$ due to Definition~\ref{def:quiver} since, by construction, $(C - \onevec) \swr A$. There is likewise a path $C \leadsto A'$ in $Q_{B}(\mathcal{T})$. Therefore, $A$ and $A'$ are connected to each other in $Q_{B}(\mathcal{T})$. Hence, any two $d$-arcs lying in a common $2d$-simplex are connected by a walk in $Q_{B}(\mathcal{T})$.

We now show that the $d$-arcs in $Q_{B}(\mathcal{T})$ which lie in different $2d$-simplices are connected with each other. Let $T, T' \in \mathcal{T}_{B}$. If one chooses points $x \in |T| \cap |B|$ and $x' \in |T'| \cap |B|$, then the line segment $\seg{x}{x'}$ connecting $x$ and $x'$ must lie entirely within $|B|$, since $|B|$ is convex. If one travels from $|T|$ to $|T'|$ along $\seg{x}{x'}$, then one runs through a series of $2d$-simplices $|T| = |T_{0}|, |T_{1}|, \dots, |T_{r}| = |T'|$ where each pair of $2d$-simplices $|T_{l - 1}|$ and $|T_{l}|$ shares a common face $|U|$ which must also intersect $|B|$. Then, by the description of the circuits of $C(n + 2d + 1, 2d)$, there must be a $d$-arc $J_{l}$ within $U_{l}$ such that $B \swr J_{l}$. Therefore, $Q_{B}(\mathcal{T})|_{T_{l - 1}}$ and $Q_{B}(\mathcal{T})|_{T_{l}}$ are connected to each other at the $d$-arc $J_{l}$, where these respectively denote the full subquivers of $Q_{B}(\mathcal{T})$ consisting of the $d$-arcs lying in $T_{l - 1}$ and $T_{l}$, respectively. Moreover, this means that $Q_{B}(\mathcal{T})|_{T}$ and $Q_{B}(\mathcal{T})|_{T'}$ are connected to each other in $Q_{B}(\mathcal{T})$. Hence $Q_{B}(\mathcal{T})$ itself is connected.
\end{proof}

\begin{remark}
Lemma~\ref{lem:connected} may also be seen quickly using an algebraic argument. \cite[Theorem 5.6]{ot} implies that $Q_{B}(\mathcal{T})$ must be the support of an indecomposable module, and so must be connected.
\end{remark}

This gives the following useful corollary, which implies that in order to check whether one can mutate a $d$-arc $A$ to a $d$-arc $B$, it suffices only to check whether the $d$-arcs adjacent to $A$ in the quiver intertwine with $B$, rather than checking all $d$-arcs for whether they intertwine with $B$.

\begin{corollary}\label{cor:adj_suff}
Let $A \in \mathcal{T}$ and $B \in \nonconsec{m}{d}$ with $A \swr B$. If there is an $A' \in \darcs{\mathcal{T}}$ with $A' \neq A$ and $A' \swr B$, then there is an $A'' \in \darcs{\mathcal{T}}$ with $A'' \swr B$, such that $A$ and $A''$ are adjacent in $Q(\mathcal{T})$.
\end{corollary}
\begin{proof}
Suppose that we are in the situation described. We know from Lemma~\ref{lem:connected} that $Q_{B}(\mathcal{T})$ is connected, and the set-up gives us that it contains at least two vertices, one of which is $A$. Hence there is a vertex of $Q_{B}(\mathcal{T})$ which is adjacent to~$A$.
\end{proof}

We can now prove the main theorem of this section.

\begin{theorem}\label{thm:retro}
Let $\mathcal{T}$ be a triangulation of $C(n + 2d + 1, 2d)$. Then a $d$-arc of $\mathcal{T}$ is mutable if and only if it is not in the middle of a maximal retrograde path in $Q(\mathcal{T})$.
\end{theorem}
\begin{proof}
For each entry $a_{i}$ in $A$, let $Z_{i} = (a_{0}, a_{1}, \dots, a_{i-1}, z_{i}, a_{i+1}, a_{i + 2}, \dots, a_{d})$ be the $d$-arc of this form in $\mathcal{T}$ such that there is an arrow $Z_{i} \rightarrow A$, if it exists. Similarly, let $B_{i} = (a_{0}, a_{1}, \dots, a_{i-1}, b_{i}, a_{i + 1}, a_{i + 2}, \dots, a_{d})$ be the $d$-arc of this form in $\mathcal{T}$ such that there is an arrow $A \rightarrow B_{i}$, if it exists.

The $d$-arc $A$ is mutable if and only if there exists $C \in \nonconsec{m}{d}$ such that $A \swr C$ but such that there is no $A' \in \mathcal{T}$ with $A' \neq A$ with $A' \swr C$. By Corollary~\ref{cor:adj_suff}, it is necessary and sufficient that we do not have $Z_{i} \swr C$ or $B_{i} \swr C$ for any $i \in \{0 , 1, \dots, d \}$. If $A \swr C$, then, since \[a_{0} < c_{0} < a_{1} < c_{1} < \dots < a_{d} < c_{d},\] we have that $Z_{i + 1}$ and $C$ do not intertwine if and only if $z_{i + 1} \leqslant c_{i}$. Similarly, $B_{i}$ and $C$ do not intertwine if and only if $c_{i} \leqslant b_{i}$.

Let $z_{i + 1} = a_{i} + 1$ if $Z_{i + 1}$ does not exist and let $b_{i} = a_{i + 1} - 1$ if $B_{i}$ does not exist. The rationale for this is that these are respectively the minimal and maximal values for $c_{i}$. Then, by the above reasoning, $A$ is mutable precisely if there exists \[C \in [z_{1}, b_{0}] \times [z_{2}, b_{1}] \times \dots \times [z_{d}, b_{d - 1}] \times [z_{0},b_{d}] .\] So $A$ is mutable if and only if this product is non-empty. But the product is non-empty if and only if $z_{i + 1} \leqslant b_{i}$ for all $i$, which is precisely the condition that none of the paths $Z_{i + 1} \rightarrow A \rightarrow B_{i}$ are retrograde.
\end{proof}

\begin{corollary}\label{cor:mut_rule}
Let $\mathcal{T}$ be a triangulation of $C(n + 2d + 1, 2d)$ and $A \in \darcs{\mathcal{T}}$ such that $A$ is not in the middle of any maximal retrograde paths. Let $Z_{i}$ and $B_{i}$ be as in the proof of Theorem~\ref{thm:retro}. Then $z_{i + 1} = b_{i}$ for all $i$ and $(b_{0}, b_{1}, \dots, b_{d})$ replaces $A$ in the bistellar flip at $A$.
\end{corollary}
\begin{proof}
We know from the proof of Theorem~\ref{thm:retro} that any element of \[[z_{1}, b_{0}] \times [z_{2}, b_{1}] \times \dots \times [z_{d}, b_{d - 1}] \times [z_{0},b_{d}]\] may replace $A$ in a bistellar flip. But, we have that the $d$-arc which can replace $A$ in a bistellar flip must be unique. Hence $z_{i + 1} = b_{i}$ for all $i$ and the unique element of the product must replace $A$ in the bistellar flip.
\end{proof}

\begin{example}\label{ex:comb_crit}
We provide examples of how one may use this criterion to identify the mutable $d$-arcs of a triangulation. We represent maximal retrograde paths using consecutive arrows of the same colour.

Considering the triangulation of $C(8,4)$ given in Figure~\ref{fig:cyc_triang}, the mutable $2$-arcs are 136, 147, and 357.

\begin{figure}
\caption{Triangulation of $C(8,4)$}\label{fig:cyc_triang}
\[
\begin{tikzpicture}[scale=1.3]

\node(135) at (0,0) {$135$};
\node(136) at (1,1) {$136$};
\node(137) at (2,2) {$137$};
\node(146) at (2,0) {$357$};
\node(147) at (3,1) {$147$};
\node(157) at (4,0) {$157$};

\draw[red,->] (135) -- (136);
\draw[blue,->] (136) -- (137);
\draw[blue,->] (137) -- (147);
\draw[green,->] (147) -- (157);
\draw[red,->] (146) -- (135);
\draw[green,->] (157) -- (146);

\end{tikzpicture}
\]
\end{figure}
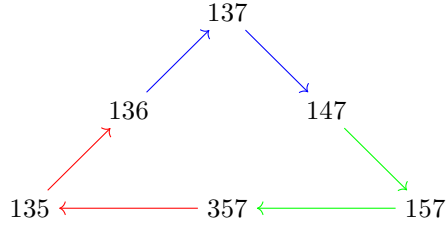

Note that maximal retrograde paths are not always of length $d$. But this is not the case. This is shown by the triangulation of $C(10,6)$ given in Figure~\ref{fig:one_short_triang}. (We use `A' to denote 10.) The mutable $d$-arcs of this triangulation are 357A, 1368, 1479.

\begin{figure}
\caption{Triangulation of $C(10,6)$}\label{fig:one_short_triang}
\[
\begin{tikzpicture}[scale=1.3,xscale=1.2]

\node(1359) at (0,0) {$1359$};
\node(1358) at (-1,-1) {$1358$};
\node(1369) at (1,-1) {$1369$};
\node(1357) at (-2,-2) {$1357$};
\node(1368) at (0,-2) {$1368$};
\node(1379) at (2,-2) {$1379$};
\node(357A) at (-2,-3) {$357A$};
\node(1479) at (2,-3) {$1479$};
\node(3579) at (-1,-4) {$3579$};
\node(1579) at (1,-4) {$1579$};

\draw[->,blue] (1358) -- (1359);
\draw[->,blue] (1359) -- (1369);
\draw[->,red] (357A) -- (1357);
\draw[->,red] (1357) -- (1358);
\draw[->,red] (1358) -- (1368);
\draw[->,green] (1368) -- (1369);
\draw[->,green] (1369) -- (1379);
\draw[->,green] (1379) -- (1479);
\draw[->,gray] (1479) -- (1579);
\draw[->,gray] (1579) -- (3579);
\draw[->,gray] (3579) -- (357A);

\end{tikzpicture}
\]
\end{figure}

One can also illustrate Corollary~\ref{cor:mut_rule}. Consider the $d$-arc 1368 in Figure~\ref{fig:one_short_triang}. This is mutable by Theorem~\ref{thm:retro}, so we can compute what $d$-arc it is exchanged for. Between 1 and 3 we must have 2 and between 6 and 8 we must have 7. Then, between 3 and 6 we must have 5 since 1368 is adjacent to 1358. Similarly, between 8 and 1 we must have 9, because 1368 is adjacent to 1369. Hence performing a bistellar flip at 1368 exchanges this $d$-arc for 2579. Observe that the retrograde-path analysis makes it easier to compute the bistellar flips of the triangulation. In the resulting triangulation, shown in Figure~\ref{fig:all_short_triang}, none of the retrograde paths are of length $d$; they are all of length $d - 1$.

\begin{figure}
\caption{Triangulation of $C(10,6)$}\label{fig:all_short_triang}
\[
\begin{tikzpicture}[scale=1.3,xscale=1.2]

\node(1359) at (0,0) {$1359$};
\node(1369) at (1,1) {$1369$};
\node(1358) at (-1,1) {$1358$};
\node(1357) at (-2,0) {$1357$};
\node(1379) at (2,0) {$1379$};
\node(357A) at (-2,-1) {$357A$};
\node(2579) at (0,-1) {$2579$};
\node(1479) at (2,-1) {$1479$};
\node(3579) at (-1,-2) {$3579$};
\node(1579) at (1,-2) {$1579$};

\draw[->,red] (357A) -- (1357);
\draw[->,red] (1357) -- (1358);
\draw[->,blue] (1358) -- (1359);
\draw[->,blue] (1359) -- (1369);
\draw[->,purple] (1369) -- (1379);
\draw[->,purple] (1379) -- (1479);
\draw[->,green] (1479) -- (1579);
\draw[->,green] (1579) -- (2579);
\draw[->,gray] (2579) -- (3579);
\draw[->,gray] (3579) -- (357A);

\end{tikzpicture}
\]
\end{figure}
\end{example}

\subsection{Mutating cut quivers}

This a rule for mutating cut quivers at sinks and sources \cite{io}, as described in Section~\ref{sect:back:cut_slice:mut}. In this section, we extend this rule to allow mutation at vertices which are not in the middle of retrograde paths, but which are not necessarily sinks or sources. In the case where the cut quiver is $Q(\mathcal{T})$ for a triangulation $\mathcal{T}$, we also describe the effect of the mutation on the triangulation $\mathcal{T}$.

For the following lemmas, we let $C$ be a cut of $\qdn$. The purpose of these lemmas is to describe the local structure of a cut quiver around a vertex which is not in the middle of a distinguished cut $(d + 1)$-cycle. We then use our knowledge of this local structure to describe the effect of mutation at that vertex.

\begin{lemma}\label{lem:adj_cyc}
If a vertex $x$ of $\qdn_{C}$ is the source of one distinguished cut $(d + 1)$-cycle and the sink of another distinguished cut $(d + 1)$-cycle, then the arrows cut out of the two cycles are of the same type.
\end{lemma}
\begin{proof}
If the arrows cut out are of different types, then they form two consecutive arrows a $(d + 1)$-cycle, which therefore has two arrows cut out. But this is a contradiction, since a cut removes precisely one arrow from each $(d + 1)$-cycle.
\end{proof}

\begin{lemma}\label{lem:onesrc_onesnk}
If a vertex $x$ of $\qdn_{C}$ is neither a source nor a sink, nor in the middle of a distinguished cut $(d + 1)$-cycle, then it is the head of precisely one arrow and the tail of precisely one arrow.
\end{lemma}
\begin{proof}
Suppose that $x$ is neither a source nor a sink, nor in the middle of a distinguished cut $(d + 1)$-cycle. Then $x$ is the head of at least one arrow $\alpha$ and the tail of at least one arrow $\beta$. Since $x$ is not in the middle of a distinguished cut $(d + 1)$-cycle, the arrows $\alpha'$ and $\beta'$ succeeding $\alpha$ and preceding $\beta$ in their respective distinguished $(d + 1)$-cycles must be cut out. By Lemma~\ref{lem:adj_cyc}, $\alpha'$ and $\beta'$ have the same type.

Suppose that we have another arrow $\gamma$ such that $x$ is the head of $\gamma$. Then the arrow $\gamma'$ which succeeds it in the distinguished $(d + 1)$-cycle must be cut out. But then by Lemma~\ref{lem:adj_cyc}, $\gamma'$, $\alpha'$, and $\beta'$ all have the same type, so $\gamma = \alpha$. A similar argument can be made for an arrow $\delta$ with tail $x$.
\end{proof}

\begin{lemma}\label{lem:vert_form}
If a vertex $A$ is the head of precisely one arrow $\alpha_{d}$ and the tail of precisely one arrow $\beta_{1}$, and not in the middle of a distinguished cut $(d + 1)$-cycle, then there is an $a_{i}$ such that for $j \notin \{i, i - 1\}$, $a_{j} + 2 = a_{j + 1}$.
\end{lemma}
\begin{proof}
The arrow $\alpha_{d}$ must be the final arrow in a distinguished cut $(d + 1)$-cycle $\alpha_{1}\alpha_{2}\dots \alpha_{d}$. The arrow $\beta_{0}$ must the first arrow in a distinguished cut $(d + 1)$-cycle $\beta_{1}\beta_{2}\dots \beta_{d}$. Let $\alpha_{0}$ and $\beta_{0}$ be the respective arrows cut out of these $(d + 1)$-cycles.

We know from Lemma~\ref{lem:adj_cyc} that $\alpha_{0}$ are $\beta_{0}$ are of the same type. We let this type be $i$. Then $\beta_{1}$ has type $i - 1$. The result then follows from Lemma~\ref{lem:space}.
\end{proof}

With these lemmas in place, we can now describe the effect of mutation on a cut quiver at a vertex which is not in the middle of a distinguished cut $(d + 1)$-cycle, but is not a sink or a source. Since the description of mutation at sinks and sources is covered in \cite{io}, as described in Section~\ref{sect:back:cut_slice:mut}, and we know that mutation at vertices in the middle of distinguished cut $(d + 1)$-cycles is not possible by Theorem~\ref{thm:retro}, this completes the description of mutation of cut quivers.

\begin{proposition}\label{prop:comp_cut_mut}
Let $\mathcal{T}$ be a triangulation of $C(n + 2d + 1, 2d)$ such that $Q(\mathcal{T})$ is isomorphic to $\qdn_{C}$ for a cut $C$. Let $A$ be a vertex of $Q(\mathcal{T})$ which is neither a sink nor a source, but is still not in the middle of any retrograde paths. Let $\mathcal{T}'$ be the result of performing a bistellar flip at $A$ in $\mathcal{T}$. Then we have the following.
\begin{enumerate}
\item In the bistellar flip $A$ is replaced by $(a_{0} + 1, a_{1} + 1, \dots, a_{i - 1} + 1, a_{i + 1} - 1, a_{i + 1} + 1, \dots, a_{d} + 1)$, where $i$ is the type of the arrow cut out of the distinguished cut $(d + 1)$-cycles at $A$.\label{op:flip_result}
\item $Q(\mathcal{T}')$ is obtained from $\mathcal{T}$ by removing from $C$ the arrows beginning or ending at $A$ and adding the arrows of $Q_{1}(\mathcal{T})\setminus C$ which begin or end at $A$.\label{op:quiv_mut}
\end{enumerate}
\end{proposition}
\begin{proof}
By Lemma~\ref{lem:onesrc_onesnk} we have that $A$ is the source of precisely one arrow and the sink of precisely one arrow. Moreover, by Lemma~\ref{lem:vert_form}, the vertex $A$ is such that there is an $a_{i}$ such that for $j \notin \{i, i - 1\}$, we have $a_{j} + 2 = a_{j + 1}$. We assume, by re-ordering, that $i = 0$. Furthermore, the distinguished cut $(d + 1)$-cycle beginning at $A$ must look like \[(a_{0}, a_{1}, \dots , a_{d}) \rightarrow (a_{0}, a_{1}, \dots , a_{d - 1}, a_{d} + 1) \rightarrow \dots \rightarrow (a_{0}, a_{1} + 1, a_{2} + 1, \dots , a_{d} + 1).\] and the distinguished cut $(d + 1)$-cycle ending at $A$ must look like \[(a_{0}, a_{1} - 1, a_{2} - 1, \dots , a_{d} - 1) \rightarrow \dots \rightarrow (a_{0}, a_{1} - 1, a_{2}, a_{3}, \dots, a_{d}) \rightarrow (a_{0}, a_{1}, \dots , a_{d}).\] It follows from Corollary~\ref{cor:mut_rule} that $A$ mutates to \[(a_{1} - 1, a_{2} - 1, \dots, a_{d} - 1, a_{d} + 1) = (a_{1} - 1, a_{1} + 1, \dots, a_{d - 1} + 1, a_{d} + 1),\] settling (\ref{op:flip_result}).

Then we have arrows in $Q(\mathcal{T}')$ given by \[(a_{1} - 1, a_{1} + 1, a_{2} + 1, \dots, a_{d} + 1) \leftarrow (a_{0}, a_{1} + 1, a_{2} + 1, \dots , a_{d} + 1),\] and \[(a_{0}, a_{1} - 1, a_{2} - 1, \dots , a_{d} - 1) \leftarrow (a_{d} + 1, a_{1} - 1, a_{2} - 1, \dots, a_{d} - 1),\] since, by assumption, there are no arrows of type $0$ ending at $(a_{0}, a_{1} + 1, \dots , a_{d} + 1)$ or $(a_{d} + 1, a_{1} - 1, a_{2} - 1, \dots, a_{d} - 1)$. Indeed, these new arrows are precisely the arrows of $C$ which begin or end at $A$. On the other hand, the arrows in $Q_{1}(\mathcal{T})\setminus C$ beginning or ending at $A$ are absent from $Q_{1}(\mathcal{T}')$. Since $a_{j} + 2 = a_{j + 1}$ for $j \notin \{i, i - 1\}$, there can be no other new arrows in $Q(\mathcal{T}')$, thus settling (\ref{op:quiv_mut}).  
\end{proof}

\begin{example}
Proposition~\ref{prop:comp_cut_mut} can be verified by mutating the left-hand triangulation in Figure~\ref{fig:mut_via_cyc} to obtain the triangulation in Figure~\ref{fig:cyc_triang}.
\end{example}

\printbibliography

\end{document}